\documentclass[12pt,letterpaper]{amsart}
\usepackage{graphicx}
\usepackage[linktocpage=true,colorlinks,citecolor=blue,linkcolor=blue,urlcolor=blue]{hyperref}
\usepackage{amssymb}
\usepackage{cite}
\usepackage{amsmath}
\usepackage{latexsym}
\usepackage[colorinlistoftodos]{todonotes}
\usepackage{amscd}
\usepackage{tikz}
\usepackage{amsthm}
\usepackage{mathrsfs}
\usepackage{url}
\usepackage[utf8]{inputenc}
\usepackage[english]{babel}
\usepackage{amsfonts}
\usepackage{mathtools}
\usepackage{comment} 
\usepackage[autostyle]{csquotes}
\usepackage{thm-restate}
\usepackage{todonotes}

\numberwithin{equation}{section}
\def\R{\mathbb R}
\def\Z{\mathbb Z}
\def\C{\mathbb C}
\def\P{\mathbb P}
\def\N{\mathbb N}
\def\E{\mathbb E}
\def\d{\mathrm d}
\def\CA{\mathcal A}
\def\CE{\mathcal E}
\def\CF{\mathcal F}

\def\ee{\varepsilon}

\def\F{\mathbb {F}}
\def\a{\alpha}
\def\D{\Delta}
\def\eps{\varepsilon}
\def\EE{\mathsf{E}}

\newtheorem{theorem}{Theorem}[section]
\newtheorem{lemma}[theorem]{Lemma}
\newtheorem{proposition}[theorem]{Proposition}

\newtheorem{corollary}[theorem]{Corollary}

\theoremstyle{remark}
\newtheorem{remark}[theorem]{Remark}

\theoremstyle{definition}

\theoremstyle{remark}

\numberwithin{equation}{section}

\begin{document}
\title{On the random Chowla conjecture}
\author{Oleksiy Klurman}
\address{Klurman: School of Mathematics, University of Bristol, United Kingdom 
}
\email{lklurman@gmail.com}

\author{ILYA D. SHKREDOV}
\address{Shkredov: Steklov Mathematical Institute,\\
ul. Gubkina, 8, Moscow, Russia, 119991
\\
and
\\
IITP RAS  \\
Bolshoy Karetny per. 19, Moscow, Russia, 127994\\
and 
\\
MIPT \\ 
Institutskii per. 9, Dolgoprudnii, Russia, 141701}
\email{ilya.shkredov@gmail.com}

\author{Max Wenqiang Xu}
\address{Xu: Department of Mathematics, Stanford University, Stanford, CA, USA}
\email{maxxu@stanford.edu}

\begin{abstract}
We show that for a Steinhaus random multiplicative function $f:\mathbb{N}\to\mathbb{D}$ and any polynomial $P(x)\in\mathbb{Z}[x]$ of $\deg P\ge 2$ which is not of the form $w(x+c)^{d}$ for some $w\in \mathbb{Z}$, $c\in \mathbb{Q}$, we have
\[\frac{1}{\sqrt{x}}\sum_{n\le x} f(P(n)) \xrightarrow{d} \mathcal{CN}(0,1),\]
where $\mathcal{CN}(0,1)$ is the standard complex Gaussian distribution with mean $0$ and variance $1.$ This confirms a conjecture of Najnudel in a strong form.\\
We further show that there almost surely exist arbitrary large values of $x\ge 1,$ such that $$|\sum_{n\le x} f(P(n))| \gg_{\deg P} \sqrt{x}  (\log \log x)^{1/2},$$
for any polynomial $P(x)\in\mathbb{Z}[x]$ with $\deg P\ge 2,$ which is not a product of linear factors (over $\mathbb{Q}$). This matches the bound predicted by the law of the iterated logarithm. Both of these results are in contrast with the well-known case of linear phase $P(n)=n,$ 
where the partial sums are known to behave in a non-Gaussian fashion and the corresponding sharp fluctuations are speculated to be $O(\sqrt{x}(\log \log x)^{\frac{1}{4}+\varepsilon})$ for any $\varepsilon>0$.
\end{abstract}

\maketitle
\section{ Introduction}
The main focus of the present paper is to take yet another look at one of the two most studied models of random multiplicative functions. Let $(f(p))_{p\ prime}$ be a sequence of independent uniformly distributed on the
unit circle $\{|z| = 1\}$ random variables. A
Steinhaus random multiplicative function is given by
$f(n)=\prod_{p^{\beta}\vert|n}f(p)^{\beta}.$
Similarly, let $(f(p))_{p\ prime}$ be a sequence of independent 
random variables taking values $\pm 1$ with probability $1/2,$ then a Rademacher random multiplicative function is given by $f(n) := \prod_{p\vert n}
f(p)$ for all $n$ which are square-free, and $f(n) = 0$ otherwise. In $1944,$ Wintner \cite{wintner} introduced Rademacher random multiplicative functions to model  the behaviour of the M\"{o}bius function $\mu(n),$ whereas Steinhaus random multiplicative functions are intended to model randomly selected Dirichlet characters $\chi(n),$ and Archemidian characters $n^{it}$ ($t\in\mathbb{R}$). We refer the reader to
\cite[Section 2]{GS} and the introduction of \cite{Helson}, \cite{Harperlargevalue}  for a meticulous overview of this subject.

A classical question of interest which attracted a lot of attention is to understand the distribution and the sizes of the partial sums $\sum_{n\le x}f(n)$. The fundamental difficulty stems from the fact that the values $f(n)$ and $f(m)$ are not independent whenever $(m,n)>1$ and thus the corresponding sums cannot be directly treated using tools for independent random variables.
\subsection{Distribution results.}
It is a natural guiding
conjecture 
that $\frac{1}{\sqrt{N}}\sum_{n\le N} f(n) \xrightarrow{d} \mathcal{CN}(0,1),$ where \enquote{$\xrightarrow{d}$} stands for convergence in distribution,
but Chatterjee suggested that this should not hold (in both Rademacher and Steinhaus cases). Chatterjee's 
conjecture (expressed in \cite{Hough}), was proved
by Harper \cite{Harper}, using an intricate conditioning argument. It is now a direct consequence of a more recent breakthrough work by Harper \cite{Helson} on Helson's conjecture that in fact
$\frac{1}{\sqrt{N}}\sum_{n\le N} f(n) \xrightarrow{d} 0.$
Interestingly, if one
restricts to several natural subsums,
Chatterjee and Soundararajan \cite{CS}, Harper \cite{Harper} and Hough \cite{Hough} established central limit theorems. It remains a deep mystery whether appropriately normalized partial sums $\sum_{n\le N} f(n)$
have a limiting distribution as $N\to\infty.$

The problem considered in this note is
motivated by the celebrated conjecture of Chowla \cite{Chowla}, which states that for the Liouville (or the M\"{o}bius) function $\lambda(n)$ and any polynomial $P(x) \in\mathbb{Z}[x],$ which is not of the form $P(x)=cg^2(x)$ for some $g\in\mathbb{Z}[x],$
\[\sum_{n\le x}\lambda(P(n))=o(x).\]
The case $k=1$ corresponds to the prime number theorem but
the general case is widely open for any polynomial with $\deg P\ge 2.$ Some remarkable progress has been recently made in the case $P(x)=\prod_{k=1}^n(a_kx+b_k)$ and $a_i,b_i\in\mathbb{Z}$ (albeit with a logarithmic weight) thanks to the combination of several works by Tao \cite{Ta1}, Matom\"{a}ki-Radziwi\l \l-Tao \cite{MRT}, Tao and Ter\"{a}v\"{a}inen \cite{TT}, and more recently by Helfgott and Radziwi\l \l \cite{HR}. A weaker statement, that $\lambda(P(n))$ changes sign infinitely often has been obtained by Cassaigne-Ferenczi-Mauduit-Rivat-S\'{a}rk\"{o}zy \cite{CFMRS}, Borwein-Choi-Ganguly \cite{BCG}, and more recently by Ter\"{a}v\"{a}inen \cite{TChowla} for a special class of polynomials $P(x)\in\mathbb{Z}[x].$

Prior to our work, we are aware of no unconditional results for Chowla's conjecture in the context of random multiplicative functions for any polynomial of $\deg P\ge 2.$ It has been previously speculated and  Najnudel \cite{Naj} conjectured that if $P(x)=x(x+1)$ (and more generally, if $P(x)=\prod_{i=1}^k(x+a_i)^{m_i}$) then the convergence in distribution
\[\frac{1}{\sqrt{N}} \sum_{n\le N}  f(n)f(n+1) \xrightarrow{d} \mathcal{CN}(0,1),\]
must hold for $f$ being a Steinhaus random multiplicative function and reformulated this conjecture {\it in terms of showing that certain family of Diophantine equations possess only trivial solutions}. Such family naturally arises while computing $2k$-th moment of the left hand side for arbitrary large values of $k\ge 1.$
Our first result is an unconditional version of a central limit theorem which works for general polynomials $P(x)\in\mathbb{Z}[x].$
\begin{theorem}\label{thm: CLT}
Let $f(n)$ be a Steinhaus random multiplicative function. Then for any polynomial $P(x)$ in $\Z[x]$ with $\deg P\ge 2$ which is not of the form $P(x) = w(x+c)^{d}$ for some $w\in \Z$,  $c\in \mathbb{Q}$, as $N\to \infty$, 
\[\frac{1}{\sqrt{N}} \sum_{n\le N}  f(P(n)) \xrightarrow{d} \mathcal{CN}(0,1).  \]
\end{theorem}
This result is optimal since in the case $P(x) = w(x+c)^{d},$ for some $w\in \Z$,  $c\in \mathbb{Q}$ we have
$\frac{1}{\sqrt{N}} \sum_{n\le N}  f(\omega(x+c)^d)) \xrightarrow{d} 0,$
after noticing that $f^d(n)$ is also a Steinhaus random multiplicative function and appealing to the results in \cite{Helson}.
It is worth mentioning that the same proof allows to deduce central limit theorems for various sparse subsums. For example, with no additional effort one could show that 
\[\frac{1}{\sqrt{\pi(N)}} \sum_{p\le N}  f(P(p)) \xrightarrow{d} \mathcal{CN}(0,1).\]

\subsection{Large fluctuations.} A classical question in probability is to understand the largest fluctuations of the sums of independent random variables. If, say, $\{\xi_k\}_{k=1}^{\infty}$ is a sequence of independent Steinhaus random variables, then according to Khintchine’s law of
the iterated logarithm, we almost surely have
\begin{equation}\label{eq: Khinchine}
\limsup_{x\to\infty}\frac{|\sum_{k\le x}\xi_k|}{\sqrt{2x\log \log x}}=1.
\end{equation}
An important feature is that \eqref{eq: Khinchine} exhibits the magnitude of the global fluctuations ( that is $\sqrt{x\log \log x}$) which is substantially larger than the expected size of the partial sums at any given point $x$ (of the order $\sqrt{x}$). 

In the case of random multiplicative functions this subject has a long and rich history. In a pioneering paper Wintner \cite{wintner} studied random Dirichlet series and in the Rademacher case 
was able to exhibit an almost sure bound $\sum_{n\le x}f(n)=O(x^{1/2+\varepsilon})$
and moreover, that one almost surely $\sum_{n\le x}f(n)=O(x^{1/2-\varepsilon})$ is false. 
This was further refined by Erd\H{o}s \cite {Erdos} to show that almost surely one has the bound $O(\sqrt{x}(\log x)^A)$ and one almost surely does not have $O(\sqrt{x}(\log x)^{-B}).$ 
In a beautiful and rather influential work, Hal\'{a}sz \cite{Hal} proved an almost sure bound 
$O(\sqrt{x}\exp(A\sqrt{\log \log x\log \log \log x}))$ 
and that one almost surely does not have
$O(\sqrt{x}\exp(-B\sqrt{\log \log x\log \log \log x}))$ for some positive constants $A, B.$
Thirty years later, Lau, Tenenbaum and Wu \cite{LTW} (see also related work \cite{Bas}) sharpened the analysis of hypercontractive inequalities in Hal\'asz’s argument, establishing an almost sure upper bound $O(\sqrt{x}(\log \log x)^{3/2+\varepsilon}).$
 On the other hand, Harper \cite{Harperomega} used Gaussian process machinery 
to study suprema of random Euler products, showing that almost surely $ O(\sqrt{x}/(\log \log x)^{5/2+\ee})$ is false. Latter results 
may be seen as approximations to the law of the iterated logarithm however quantitatively substantially weaker. 
In a recent breakthrough, answering
a question of Hal\'{a}sz and proving an old conjecture of Erd\H{o}s, Harper \cite{Harperlargevalue} showed that if $f(n)$ is a Steinhaus (or Rademacher) random multiplicative
function, then almost surely $|\sum_{n\le x}f(n)|\ge \sqrt{x}(\log \log x)^{1/4-\varepsilon}$ holds for sequence of arbitrary large values of $x\ge 1.$ Remarkably, this furnishes the first bound that grows faster than $\sqrt{x}$ and moreover the exponent $1/4$ is speculated to be sharp (see also \cite{Helson}, \cite{Mostra}).

We establish a lower bound of the size $\sqrt{x\log \log x}$ matching the one predicted by the Khinchine's type law of the iterated logarithm. 
\begin{corollary}\label{Khinchine}
Let $f(n)$ be a Steinhaus random multiplicative function. Then for any polynomial $P(x)$ in $\Z[x]$ with $\deg P\ge 2$ which is not a product of linear factors (over $\mathbb{Q}$), there almost surely exists arbitrarily large $x$ such that
\begin{equation}\label{eqn: large value}
    |\sum_{n\le x} f(P(n)) | \gg_{d}  \sqrt{x\log\log x}.
\end{equation}
\end{corollary}
In fact, we prove a more general local version and then apply the standard Borel-Cantelli type argument to deduce Corollary \ref{Khinchine}.

\begin{theorem}\label{thm:largevaluelocal} 
    Let $f(n)$ be a Steinhaus random multiplicative function and let $P(x)$ be a polynomial in $\Z[x]$ with $\deg P\ge 2$ which is not a product of linear factors (over $\mathbb{Q}$). Then uniformly for all large $X$, 
    \begin{equation}\label{eqn: large value local}
      \max_{X\le x \le X^{(\log X)^{2}}} \frac{1}{\sqrt{x}}|\sum_{n\le x} f(P(n))| \ge c_d \sqrt{\log \log X}
    \end{equation}
     with probability $1-O(\frac{1}{(\log \log X)^{0.02}})$ for some fixed $c_d>0$ depending on the degree $d=\deg P$.
\end{theorem}
We conclude this section by mentioning that in the deterministic case, a well-known conjecture of Gonek \cite{Ng} predicts the sharp upper bound $\sum_{n\le x}\mu(n)=O(\sqrt{x}(\log\log\log x)^{5/4}).$ In view of Theorem \ref{thm: CLT} and Corollary \ref{Khinchine} it seems reasonable to expect that the largest fluctuations of the Chowla type sums $\sum_{n\le x}\mu(P(n))$ are of the order $\sqrt{x\log \log x}$ for any admissible polynomial $P(x)\in\mathbb{Z}[x]$ of $\deg P\ge 2.$

\subsection{Outline of the proofs.} 
A standard point of departure in establishing central limit theorems is a computation of higher (integral) moments: $$\mathbb{E}|\sum_{n\le x}f(P(n))|^{2k}=\sum_{n_i,m_i\in [N]}\text{I}_{P(n_1)P(n_2)\dots P(n_k)=P(m_1)P(m_2)\dots P(m_k)}.$$
Latter naturally leads to a consideration of the {\it higher multiplicative energies} of the polynomial images, the question interesting on its own right (see Section 2 for the discussion). This seems to be a rather difficult problem as far as general polynomials $P\in\mathbb{Z}[x]$ are concerned for any $k\ge 3.$

To overcome this obstacle and prove Theorem \ref{thm: CLT}, we take advantage of the crucial feature that the partial sums $\sum_{n\le x}f(P(n))$ exhibit  structure of a martingale difference sequence. Such an observation has been previously utilized by several authors including Harper \cite{Harper}, Lau, Tenenbaum, Wu \cite{LTW} in the context of studying non-Gaussian behaviour of $\sum_{n\le N}f(n).$
After applying  McLeish's martingale central limit theorem (which we will recall in Section~\ref{Sec: CLT}), fortunately, the case $k=2$  suffices to accomplish our modest task.
To this end for subsets  $A\subseteq \R$, we introduce a multiplicative energy \cite{TV}  of the set
	$
	\EE^{\times}(A):= 
	|\{ (a_1,a_2,a_3,a_4) \in A~:~ a_1a_2=a_3a_4 \}| \,
    $
and prove the following.
	\begin{proposition}\label{energy}
	Let $N\ge 1$ be a positive integer and $P(x)\in\mathbb{Z}[x]$ be a polynomial of degree $d\ge 2$ and $P(x)\ne w (x+c)^d$ for any $w\in \Z$, $c\in \mathbb{Q}$. Then, for $d>2$ we have bounds
	\[
	\EE^{\times}(P([N]))= 2N^2 + O( N^{2-\frac{1}{2(2d-1)}+o_d(1)})
	\]
	and for $d=2$
	\[
	\EE^{\times}(P([N]))= 2N^2 + O(N^{5/3+o(1)}).\]
	\end{proposition}
Proposition \ref{energy} will be immediately deduced from a more general Theorem ~\ref{thm: energy}, with the key input in the proof coming from the use of a celebrated result of Bombieri--Pila \cite{BP} bounding the number of integral points on curves.

The proof of Theorem \ref{thm:largevaluelocal} heavily relies on the two probabilistic results (in the form established in ~\cite{Harper}): normal comparison and normal approximation results (Lemma \ref{lem: normal approx} and Lemma \ref{lem: normal comparison} respectively).
Roughly speaking, Harper establishes a multivariate Gaussian approximation for the sums $\sum_{X< p\le x}f(p)\sum_{n\le x/p}f(n)$ conditional 
on all the values $(f(p))_{p\le X}$, sampled at a well spaced sequence of ($\asymp \log X$) points $X^{8/7}\le x\le X^{4/3},$ thus making the inner sums fixed. The main bulk of the work then goes into analyzing the sizes of conditional variances and covariances using techniques from multiplicative chaos. In our case, we instead proceed differently and choose the set of primes to condition on more judiciously (in the spirit of greedy algorithm). The key consequence of such conditioning argument together with Proposition \ref{energy} is that the \enquote{essential} parts of the random sums at different scales $x$ become independent with conditional variance being roughly of the same size $\sqrt{x}$ as unconditional one. This might explain why our bounds match those predicted by \eqref{eq: Khinchine}.
The main arithmetic input we use here is that for any polynomial $P \in \Z[x],$ which is not a product of linear factors, the set of $n\le X$ with the largest prime factor
$P^{+}(P(n)) \gg_d n,$
has positive density.
In order to efficiently apply Borel--Cantelli type argument we need at least $(\log X)^{\delta}$ sample points on the interval $X\le x\le X^{\log^2 X}$ for some $\delta>0$ and so we apply quantitative form of the latter fact recently established by Maynard and Rudnick in \cite{MaynardRudnick} together with a simple upper bound sieve (which crucially wins the desired power of $\log X$).
\subsection{Organization of the paper and future work.}
We prove Theorem~\ref{thm: CLT} in Section~\ref{Sec: CLT} with the crucial energy bounds deferred to be proved in Section~\ref{Sec: energy}. In Section~\ref{Sec: large value}, we prove Theorem~\ref{thm:largevaluelocal} and Corollary~\ref{Khinchine}. The situation with Rademacher random multiplicative functions is more delicate. With additional effort, the methods of the present paper also work in that case for $P$ belonging to a wide class of polynomials. However for the case of $P$ with $\deg P\ge 3$ even existence of positive proportion of square-free values of $P(n)$ is only known under assumption of the ABC conjecture thanks to the work of Granville \cite{ABC}. We have decided to keep the presentation here relatively simple focusing on the main ideas rather than generality of the results. In the future work, we shall return to the study of the Rademacher case (both unconditionally and conditional on the ABC conjecture).
\subsection*{Acknowledgment.}
The authors are very grateful to Adam Harper and Alexander Mangerel for the insightful comments on an earlier draft of this manuscript. O.K. and I.S. would like to thank Max Planck Institute for Mathematics (Bonn) for providing excellent working conditions and support. O.K. would also like to express his gratitude to Mittag-Leffler Institute for mathematical research for providing stimulating working environment and support.
 M.W.X. is supported by the Cuthbert C. Hurd Graduate Fellowship in the Mathematical Sciences (Stanford).

\section{Proof of Theorem~\ref{thm: CLT}}\label{Sec: CLT}

Our main tool in this section is the following version of the central limit theorem due to McLeish \cite{McLeish}. The idea of using this result in the context of random multiplicative functions can be traced back to \cite{Harper} (see also \cite{SoundXu}). 

\begin{lemma}[McLeish]\label{Thm:McLeish}
		For $n \in \N$, suppose $k_n \in \N$, and that $X_{i,n}, 1\le i\le k_n$ is a martingale difference sequence on $(\Omega, \mathcal{F}, (\mathcal{F}_{i,n})_i, \P)$ where $\mathcal{F}_{i,n}$ is the sigma algebra generated by all $\{X_{j, n}: j\le i \}$ i.e., for any $i\ge 1$
		\[\E(X_{i+1,n}| \mathcal{F}_{i,n}) = 0. \]
		Write $S_n: = \sum_{i\le k_n} X_{i,n}$, and suppose that the following conditions hold:
		\begin{enumerate}
			\item (normalized variances) $\sum_{i\le k_n} \E X_{i,n}^{2} \to 1$ as $n\to \infty$. 
			\item (Lindeberg condition) 
$\sum_{i\le k_n} \E(|X_{i,n}|^{4}) = o_{n\to\infty}(1)$.
			\item (cross term condition) 
			$\limsup_{n\to \infty} \sum_{i,j\le k_n, i\neq j} \E X_{i,n}^{2}X_{j,n}^{2} \le 1 $.
		\end{enumerate}
		Then $S_n$ converges in distribution to a normal random variable with mean $0$ and variance $1$. 
	\end{lemma}
\begin{proof}[Proof of Theorem~\ref{thm: CLT} assuming Proposition~\ref{energy}]
We write 
\[f(p) = e(\theta_p)\]
where $(\theta_p)_p$ is a sequence of i.i.d., taking values in $[0,2\pi]$ uniformly. One can define $\theta_n$ in a natural way by using additivity. It suffices to show that
$X : =X(N)=\mathfrak{Re} \frac{1}{\sqrt{N}} \sum_{n\le N }f(P(n))$ and $Y : =Y(N)=\mathfrak{Im} \frac{1}{\sqrt{N}} \sum_{n\le N }f(P(n))$
both converge to standard normal distribution each with mean zero and variance $1/2$ independently. By definition, it follows that each of them has mean zero and $\E X(N)Y(N) = 0$, i.e., these variables are uncorrelated. The variance of $X(N)$ (and similarly $Y(N)$) can be computed as 
\[\E |X(N)|^{2} =\frac{1}{\sqrt{N}} \sum_{n_1, n_2\le N} \E \cos (\theta_{P(n_1)}) \cos (\theta_{P(n_2)}) = \frac{1}{2}. \]

To show the independence, as $X(N)$ and $Y(N)$ are uncorrelated, it is enough to show that $X(N)$ and $Y(N)$ have jointly Gaussian limiting distribution, i.e., $aX(N)+b Y(N)$ is Gaussian for any choice of $a, b \in \R$.  Notice that 
\[aX(N)+bY(N) =\sqrt{\frac{{a^{2} + b^{2}}}{{N}} } \sum_{n\le N} \cos(\theta_{P(n)} +c), \]
where $c= \arctan(b/a).$ Since this expression is of the same shape as the one for $X(N),$ we shall only focus on the latter case.

To prove that $X(N)$ converges to a Gaussian distribution with mean zero and variance $1/2$, we invoke Lemma~\ref{Thm:McLeish}.
Let $P^{+}(m)$ be the largest prime factor of a positive integer $m$ and  let
\[M_p(N) : =  \mathfrak{Re} \frac{1}{\sqrt{N/2}} \sum_{n\le N: P^{+}(P(n)) = p} f(P(n)).\]
Let $\mathcal{F}_q$ be the sigma algebra generated by all random variables indexed by the primes of size at most $q$. We observe that
\[\E (M_p(N) | \mathcal{F}_q: q<p ) = 0,\]
yielding that $(M_p(N))_p$ form a martingale difference sequence. It thus suffices to check that all conditions of Lemma~\ref{Thm:McLeish} are satisfied. The normalisation condition is trivial. To check the Lindeberg condition, we need to show that
\[\sum_{p\ll N^{d}} \E |M_p(N)|^{4} = o(1). \]
Here the sum runs over all $p\ll N^d$ that divides some value of $P(n),$ for $n\le N.$ 
The above expression can be rewritten in the form
\[\sum_{p\ll N^{d}} \sum_{\substack{n_i\le N\\ P^{+}(P(n_i)) = p}} \E \cos (2\pi \theta_{P(n_1)}) \cos (2\pi\theta_{P(n_2)}) \cos (2\pi\theta_{P(n_3)})\cos (2\pi\theta_{P(n_4)}) = o(N^{2}). \]
This is equivalent to showing that, by using the orthogonality
\[\sum_{p \ll N^{d}} \# \{ n_1, n_2, n_3, n_4 \in [N]^{4}: P^{+}(P(n_i)) = p, P(n_1)P(n_2) = P(n_3)P(n_4)\} =o(N^{2}). \]
Using Proposition~\ref{energy}, we conclude that the off-diagonal contribution is $o(N^{2})$. To estimate our diagonal contribution, we distinguish between three ranges:  $p \ll  \log \log N$, $ \log \log N \ll p \le N $ or $p> N$.\\
If $p \ll \log \log N$, then the contribution in this case is 
\[\ll_d (\log N)^{O(\log \log N)} = o_d(N^{2}).\]
If $ \log \log N \ll p \le N$, notice that the number of $n\le p$ such that $p|P(n)$ is at most $d=\deg P$ for each fixed $p$ and consequently, the number of diagonal solutions is at most 
\[\ll_d  \sum_{\log \log N \ll p \le N} \left \lceil \frac{N^{2}}{p^{2}} \right \rceil  = o_d(N^{2}). \]
If $p\ge N$ we notice that for each fixed $n\le N$ there are at most $O_d(1)$ primes $p\ge N$ with $p\vert P(n)$ and therefore there is in total $O_d(N)$ number of pairs $(p,n)$ such that $p\vert P(n)$ and $p\ge N$. Combining with the fact that for each $p\ge N$ there are at most $O_d(1)$ integers $n$ such that $p\vert P(n)$, it follows that the number of diagonal solutions in this regime is  $O_d(N)$ which is negligible.

The cross term condition is given by \[\limsup_{N \to \infty} \sum_{p\neq q \le N} \E M_p(N) ^{2} M_q(N)^{2} = 1,   \]
which is 
clearly equivalent to showing that \[ \sum_{p\neq q \le N}  \sum_{\substack{P^{+}(P(n_1))= P^{+}(P(n_2)) =p\\ P^{+}(P(n_3)) =P^{+}(P(n_4))=q }} \E  \cos(2\pi \theta_{P(n_1)})\cos(2\pi \theta_{P(n_2)})\cos(2\pi \theta_{P(n_3)})\cos(2\pi \theta_{P(n_4)}) =  \left(\frac{1}{4}+o(1)\right) N^{2}.  \]
Using orthogonality and dropping the condition $p\neq q,$ the above sum equals to $1/8$ times the number of solutions to
\[ P(n_1)P(n_3) =P(n_2)P(n_4), \quad n_i \in [N]. \]
The main term $N^{2}/4$ comes from the diagonal contribution $2N^{2}$ and the essential task of bounding the off-diagonal contribution once again follows from Proposition~\ref{energy}. 
\end{proof}

\section{Energy bounds and paucity phenomena}\label{Sec: energy}
The main 
purpose 
of this section is to give a proof of Theorem \ref{thm: energy} which  concerns the paucity phenomenon of polynomial sequence and directly implies 
Proposition \ref{energy}. 
Our task is to calculate the number of integral points on the variety 
\begin{equation}\label{def:variety-}
    \mathcal{V}_P = \{ (x_1,x_2, x'_1, x'_2) \in [N] ~:~ P(x_1) P(x_2) =  P(x'_1) P(x'_2) \} \,
\end{equation} 
 and more generally, we consider the variety 
\begin{equation}\label{def:variety}
    \mathcal{V}_P = \{ (x_1,\dots, x_k, x'_1, \dots, x'_k) \in \F^{2k} ~:~ P(x_1) \dots P(x_k) =  P(x'_1) \dots P(x'_k) \} \,,
\end{equation} 
where $\F$ is an arbitrary field and $k\ge 2.$
We aim to obtain a paucity result, that is, the number of the \enquote{non--trivial}  solutions is negligible relative to the \enquote{trivial} ones.
There are two basic questions concerning $\mathcal{V}_P$. The first one is related to the definition of the 
\enquote{trivial} solutions. 
The points $(x_1,\dots, x_k, x'_1, \dots, x'_k)$ with $\{ x_1,\dots,x_k \}= \{x'_1,\dots,x'_k\}$ clearly belong to $\mathcal{V}_P$ and  one can hope that those contribute the main term to $\EE^\times (P[N]).$ The other natural choice comes from the \enquote{trivial} curves lying on \eqref{def:variety} which are of the form $P(x_i) = P(x'_j)$, $i,j \in [k]$. 
To this end, one can show that if the curve $P(x)-P(y) = 0$ is irreducible, then it contains negligible number of points. The question of its reducibility is more subtle and it is known \cite[Theorem 1]{Fried} that the polynomial $
    \phi (x,y) := \frac{P(x)-P(y)}{x-y} 
$
    is absolutely irreducible unless $P(x)$ is decomposable, that is of the form $h(r(x))$ for some polynomials $h(x), r(x).$ See \cite{DLS}, \cite{DS}, \cite{Schinzel} for further discussion of this notion. 
    Other examples come from the families $P(x) = ax^d +b$ and $P(x)= cT_d(x)$, where $T_d (x)$ is the Chebychev polynomial of the first kind. 
On the other hand, if $P(x)=h(r(x))$, then
    solutions $r(x_i) = r(x'_j)$ can be 
    treated 
    as \enquote{trivial} and $r(x_i) = r(x'_j)$ and $x_i = x'_j$ have approximately the same number of solutions.

The second question is concerned with the low--dimensional subvarieties of $\mathcal{V}_P$. 
Typically, such subvarieties contain the main mass of the solutions (see \cite{H-B_rational}). Fortunately, we will be able to get away by considering just one--dimensional subvarieties and consequently we need to understand lines belonging to $\mathcal{V}_P$.

Let $\F$ be an algebraically closed field, $k \ge 2$ be an integer,  $P(x) \in \F[x]$, $P(x) \neq \omega (x+c)^d$ for any $\omega, c\in \F$.
    Let $\mathcal{Z}_P$ denote the set of distinct zeros of $P(x)$. 
    Since $P(x) \neq \omega (x+c)^d$, it follows that $|\mathcal{Z}_P|>1$. 
    Let $l_1,\dots, l_k, l'_1,\dots, l'_k$ be non--vertical and non--horizontal linear transforms and suppose that  
\begin{equation}\label{def:lines}
    \{ (l_1 (t),\dots, l_k (t), l'_1 (t), \dots, l'_k (t) ) ~:~ t\in \F \} \subseteq \mathcal{V}_P \,.
\end{equation}
    We are interested in describing nontrivial families of  lines $L = \{ l_1,\dots, l_k\}$, $L' = \{ l'_1, \dots, l'_k \}$, that is $L\neq L'$, satisfying \eqref{def:lines}. 
    
{\bf Example 1.}
  We call polynomial $P_\beta (x)$ a generalized even polynomial if $P_\beta (x)=g(x-\beta/2)$, where $g(x)$ is an even polynomial.
    Clearly $P_\beta (x) = P_\beta (\beta-x)$ which produces large family of nontrivial lines. In our case, we confine ourselves with positive variables and thus such an obstacle could be easily treated.
 
{\bf Example 2.} 
    Let us turn our attention to the following construction from \cite[Section 18.2.2]{Prasolov}. Let $d_0 =\deg P \ge 3$, $\a^{d_0} = 1$, $\a\neq \pm 1$ and let 
\begin{equation}\label{def:circle}
    P_{\a,\beta} (x)= a_0 \left( x + \frac{\beta}{\alpha-1} \right)^{d_0} + c, \quad \mbox{ where } a_0 \in \F\setminus \{0\} \mbox{ and } c\in \F \,.
\end{equation} 
    Then $P_{\a,\beta} (x) = P_{\a,\beta} (\a x + \beta)$ and upon taking different pair $\a_*$, $\beta_*$, $\a^d_* = 1$ such that $\frac{\beta}{\alpha-1} = \frac{\beta_*}{\alpha_*-1}$, we obtain $P_{\a,\beta} (x) =P_{\a_*,\beta_*} (x).$ Consequently, for any $t\in \F$ one has 
\[
    P_{\a,\beta} (t) P_{\a_*,\beta_*} (t) = P_{\a,\beta} (\a t + \beta) P_{\a_*,\beta_*} (\a_* t + \beta_*) \,.
\]
    Moreover, one can further consider $P(x) := h(P_{\a,\beta} (x))$ for any $h\in \F[x]$ to increase the number of possible ``nontrivial subvarieties".

    To overcome the difficulties mentioned above,
    we formulate a simple finiteness result on lines \eqref{def:lines} (see Lemma \ref{l:lines_in_V} below).
    We begin by introducing generalizations of polynomials \eqref{def:circle}.
    Let $\mathcal{I}$ be the family of all possible products of such polynomials,
    \begin{equation}\label{eqn: I}
     \mathcal{I}: = \{P_1P_2\cdots P_j: j\in \mathbb{N}, P_i~\text{satisfy}~\eqref{def:circle}~\text{for all $1\le i\le j$} \}.   
    \end{equation}
     We observe that $r\in \mathcal{I}$ if and only if the set $\mathcal{Z}_{r_{}}$ is a shift of a union of some concentric regular  polygons.  Notice that $\mathcal{I}$ contains all generalized even polynomials. 
    We now consider polynomials of the form $a_0 \prod_{j=1}^d (x+b_0-\rho^j)$ where $\rho$ is {\it not} a root of unity  and, more generally define the family
\begin{equation}\label{def:two_circles}
    a_0 \prod_{i=1}^{d_1} (x+b_1-\rho_1^i) \dots \prod_{i=1}^{d_s} (x+b_s-\rho_s^i) \cdot r(x) \,,
\end{equation}
    where $s\ge 0$ is an integer, $r\in \mathcal{I}$, $d_1+\dots+ d_s + \deg (r) = \deg(P)$, $a_0\in \F\setminus \{0\}$, $b_1,\dots, b_s \in \F$ and $\rho_1, \dots, \rho_s \in \F$ are {\it not} the roots of unity.
    One can check that for any polynomial $P$ belonging to the family \eqref{def:two_circles}, the set $\mathcal{Z}_P$ consists of the union of at most $s$ shifts of geometric progressions and
    a shift of 
    a union of some concentric regular  polygons. 
    Finally,
    let $\mathcal{L} (\mathcal{Z}_P)$ be the set of all lines, generated by $\mathcal{Z}_P \times \mathcal{Z}_P$.
    Since an arbitrary line is determined by any two points of $\mathcal{Z}_P \times \mathcal{Z}_P$ we have that  $|\mathcal {L} (\mathcal{Z}_P)| \le \binom{|\mathcal{Z}_P|^2}{2}$, and the line $l(t)=t$ always belongs to $ \mathcal{L}(\mathcal{Z}_P)$.

\begin{lemma}
Let $\F$ be an algebraically closed field, and  $P(x) \in \F[x]$ such that $P(x) \neq \omega (x+c)^d$ for any $\omega, c\in \F$ and $d\ge 2$. Let $\mathcal{Z}_P$ be the set of distinct zeros of $P(x)$ in $\F$ with $|\mathcal{Z}_P| > k\ge 2$ and the lines $l_1,\dots, l_k, l'_1, \dots, l'_k$ satisfy \eqref{def:lines}.  
    Then, 
    \\
    $1)~$ For any $i\in [k]$ there exists $j\in [k]$ such that $l'_i \circ l^{-1}_j \in \mathcal {L} (\mathcal{Z}_P)$.\\
    $2)~$ If $P\notin \mathcal{I}$ and $l'_i \neq l^{}_j$, then 
    there exists $j_* \in [k]$, $j_* \neq j$ such that the graph of $l'_i \circ l^{-1}_{j_*}$ intersects $\mathcal{Z}_P\times \mathcal{Z}_P$.\\
    $3)~$ Let $s$ be an integer, $s\ge k-1$, $l'_i \neq l^{}_j$ and suppose that $P(x)$ is not of the form \eqref{def:two_circles}.
    Then  there exists  $j_* \in [k]$, $j_* \neq j$ such that $l'_i \circ l^{-1}_{j_*} \in \mathcal {L} (\mathcal{Z}_P)$.
\label{l:lines_in_V} 
\end{lemma}
\begin{proof}
    From the definition of $\mathcal{V}_P$ it follows that for any $i\in [k],$ if  $l'_i (t) \in \mathcal{Z}_P$ for some $t$, then there exists $j\in [k]$ with 
    $l_j (t) \in \mathcal{Z}_P$. 
    Hence $(l'_i \circ l^{-1}_j ) (z_1) = z_2$ for some $z_1,z_2 \in \mathcal{Z}_P$ and the graph of $l'_i \circ l^{-1}_j$ intersects $\mathcal{Z}_P\times \mathcal{Z}_P$. 
    Since $|\mathcal{Z}_P| > k,$
 we have $|\mathcal{Z}_P|$ distinct zeros of $P$ and $|\mathcal{Z}_P|$ pairs of lines $(l'_i, l_j)$, $j\in [k]$, by the pigeonhole principle there is a line $l^{-1}_j$ such that the graph of  $l'_i \circ l^{-1}_j$ intersects $\mathcal{Z}_P\times \mathcal{Z}_P$ in at least two points and hence it belongs to $\mathcal{L} (\mathcal{Z}_P)$. This concludes the proof of 1).\\ 
    To prove 2) and 3), we
    consider $l:= l'_i \circ l^{-1}_j$ such that $l(t)$ is not an identical map. 
    We begin by considering all transformations with $\a \mathcal{Z}_P + \beta = \mathcal{Z}_P$.
    By shifting one can assume that $\a \mathcal{Z}'_P = \mathcal{Z}'_P$ and consequently the set $\mathcal{Z}'_P := \mathcal{Z}_P - \beta/(\a-1)$ is a geometric progression with step $\a$.
    Since $|\mathcal{Z}'_P| = |\mathcal{Z}_P|>1$, it follows that $\a$ is a root of unity and by shifting again, if necessary, we arrive at the conclusion that $P\in \mathcal{I}$. Since by our assumption, $P \notin \mathcal{I}$ we have that there exists $j_* \in [k]$, $j_* \neq j$ such that the graph of $l'_i \circ l^{-1}_{j_*}$ intersects $\mathcal{Z}_P\times \mathcal{Z}_P$. More generally, we have seen that $\mathcal{Z}'_P$ is a union of $\a$--invariant sets and a non--invariant part. 
    Split the non--invariant part as a union of non--invariant geometric progressions with step $\a$. 
    Since $P(x)$ is not of the form \eqref{def:two_circles}, the number of such progressions must be at least $s+1$.
Consequently, $|l(\mathcal{Z}_P) \cap \mathcal{Z}_P| < |\mathcal{Z}_P|-s$. 
    By our assumption $s\ge k-1$ and applying the pigeonhole principle again
    we find $j_* \in [k]$, $j_* \neq j$ such that $l'_i \circ l^{-1}_{j_*} \in \mathcal {L} (\mathcal{Z}_P)$.  
    This completes the proof. 
\end{proof}

\bigskip 
It will be convenient to formulate our energy results with variables constrained to certain arithmetic progressions. To this end, for positive numbers $q<N/2$ and non--negative $0\le a<q$ we let $[N]_{a,q}$ to denote the set of $x\in [N]$ such that $x\equiv a \pmod q$.
In particular, for $q=1$, $a=0$ we have $[N]_{a,q}=[N].$ 
Let $\d(a)=1$ for $a=0$ and $\d(a)=0$ otherwise.

\begin{theorem}\label{thm: energy}
    Let $P(x) \in \Z[x]$ with $\deg (P) = d \ge 2$,  and let $N$,
    $q<2^{-1} N^{\frac{1}{2(1+\d(a))}}$
    be positive integers. If $P(x) \neq \omega (x+c)^d$ for any choice of $\omega\in \Z$, $c\in \mathbb{Q}$, then for $d>2$
\begin{equation}\label{f:E(p)}
    \EE^\times (P([N]_{a,q}) )  - \frac{2N^2}{q^2} \ll \frac{N^{2-\frac{1}{2(2d-1)} + o_d(1)}}{q^{2+\frac{\d(a)}{2d-1}}} \,, 
\end{equation}
    and for $d=2$ the following holds 
\begin{equation}\label{f:E(p)_2}
    \EE^\times (P ([N]_{a,q}) )  - \frac{2N^2}{q^2} \ll \frac{N^{5/3+o_d(1)}}{q^{2+\frac{\d(a)}{3}}} \,. 
\end{equation}
\end{theorem}
We remark that the terms $2N^2/q^2$ correspond to the diagonal solutions and thus Theorem \ref{thm: energy} yields a power saving for the off-diagonal contribution. 
The main ingredient in our proof is the following  celebrated result due to Bombieri and Pila \cite[Theorem 5]{BP}. 

\begin{lemma}[Bombieri--Pila]\label{thm: BP}
  Let $\mathcal{C}$ be an absolutely irreducible curve (over the rationals) of degree $d\ge 2$ and $N\ge \exp(d^{6})$. Then the number of  integral points on $\mathcal{C}$ and inside a square $[0, N] \times [0, N]$ does not exceed 
  \[ N^{\frac{1}{d}} \exp( 12 \sqrt{d \log N \log \log N}). \]
\end{lemma}
Having Lemma \ref{l:lines_in_V} and Lemma \ref{thm: BP} at our disposal, we are ready to obtain the main result of this section.

\begin{proof}[Proof of Theorem~\ref{thm: energy}]
Let $\tau := \max_{n\in [|P(N)|^2]} \tau (n) = N^{o_d(1)}$ where $\tau(n)$ is the number of divisors of $n$.

{\it Preliminary reduction.}
We begin with the following simple observation.
By our assumption  $P(x) \neq \omega (x+c)^d$ for any $\omega\in \Z$, $c\in \mathbb{Q}$
    and hence performing a rational change
    of variables
    one can assume  that 
    $P(x) =x^d + g(x)$, where $g  \in \mathbb{Q}[x]$, $\lambda \neq 0$ is the leading coefficient of $g$ and $\deg g = m \le d-2$. 
    Now our variable $x$ runs over a rational shift of $[N]_{a,q}$, say, $\frac{u[N]_{a,q} + v}{w}$ with $u, v, w \in \Z$ and $u, w> 0$. We multiply $P(x)$ by $w^{d}$ which clearly does not change the multiplicative energy and thus we may assume that $x \in u[N]_{a,q}+v$. 
    Since $u[N]_{a,q} +v\subset{[uN]_{au+v,uq}}$
   it suffices to estimate the off diagonal contribution in $\EE^\times (P([uN]_{au+v,uq})).$

{\it Main argument.} In view of the above, changing $q\to uq$ and $a\to au+v$ if necessary, we need to estimate the number of solutions $x,y,X,Y\in[N]_{a,q}$ to the equation 
\begin{equation}\label{tmp:1-}
    (x^d + g(x))(y^d + g(y)) = (X^d + g(X))(Y^d + g(Y)) 
\end{equation} 
    or, in other words, 
\begin{equation}\label{tmp:1}
    (XY)^d - (xy)^d =   
    (x^d g(y) + y^d g(x) + g(x) g(y)) - (X^d g(Y) + Y^d g(X) + g(X) g(Y)) \,.
\end{equation}
    The choice $\{x,y\} = \{X,Y\}$ corresponds to $\frac{2N^2}{q^2} + O(N/q)$ solutions of the last equation and thanks to the condition     $q<2^{-1} N^{\frac{1}{2(1+\d(a))}}$ we see that the term $O(N/q)$ is negligible when compared to $N^{2-\frac{1}{2(2d-1)} + o_d(1)} \cdot q^{-2-\frac{\d(a)}{2d-1}}$.
    Now let $\Delta \le N/q$ be a parameter to be chosen later. 
   We may assume that all variables take the form $qk+a$, where $k\ge \D.$ Indeed, otherwise we have in total at most $4\D \tau N/q$ solutions. 
    
    Let $s=XY-xy$ and $t=X+Y \in [2,2N]$ and notice that $s$, $t-2a$ are divisible by $q$
    (if $a=0$, then $s$ is divisible by $q^2$). 
    If $s=0$ and $g=\lambda x^{m}$, $m\le d-2 $  (the case of constant $g$  corresponds to $d=2$), then we obtain just trivial solutions of our equation. 
Without loss of generality we first assume that $s>0$ and write    
\begin{equation}\label{eq:XY-xy_s}
    (XY)^d - (xy)^d = (xy+s)^d - (xy)^d = 
    s \sum_{j=0}^{d-1} \binom{d}{j} (xy)^j s^{d-j-1} \,.
\end{equation}
    From \eqref{tmp:1} and the fact that $s>0$ it follows that $|s| (\D q^{})^{2d-2} \ll N^{d+m} \le N^{2d-2}$ and hence 
\begin{equation}\label{cond:s} 
    |s| \ll (N/\D q^{})^{2d-2} \,.
\end{equation}
    From the binomial formula we get 
\[
    (\a+\beta)^n = \a^n + \beta^n + \sum_{j=1}^{n-1} \binom{n}{j} \a^j \beta^{n-j} 
\]
and by induction we can write symmetric polynomial as $\a^n + \beta^n = f_n (\a \beta, \a+\beta)$, where $f_n \in \Z[z,w]$, $f_n (z,w) = w^n + \tilde{f}_n (z,w)$, $\deg_z \tilde{f}_n = \deg_w \tilde{f}_n = n-2$, $n>2$
    and $f_2 (z,w) = w^2-2z$ for $n=2.$ We now fix variables $s, t$ and define 
    \[P_{s,t} (xy) : =  - (X^d g(Y) + Y^d g(X) + g(X) g(Y)) - (  (XY)^d - (xy)^d ).\]
Using \eqref{eq:XY-xy_s}, we get $$P_{s,t} (xy) = (2\lambda' - sd) (xy)^{d-1}+ \tilde{P}_{s,t} (xy)$$ and  $\deg \tilde{P}_{s,t} \le d-2,$
    $\lambda'=\lambda$ if $m=d-2$ and zero otherwise.  
    Hence \eqref{tmp:1}, with $s,t$ being fixed takes the form
\begin{equation}\label{f:G+P}
    \sigma:= G(x,y) + P_{s,t} (xy)   = 0 \,,
\end{equation}
where $G(x, y) : =   x^d g(y) + y^d g(x) + g(x) g(y) $. 

We first consider the case that polynomial $G+P_{s,t}$ has no linear factors over $\C$. If $G+P_{s,t}$ is absolutely irreducible, then by Lemma~\ref{thm: BP} we have at most $N^{1/d+o(1)}$ solutions in $x,y$ for $d>2$.  For $d=2$ it is easy to check directly that the number of solutions is $N^{o(1)}$. In general, considering absolutely irreducible factors of $G+P_{s,t}$ and recalling that there are no linear factors, we apply Lemma~\ref{thm: BP} with $d=2$ to bound the total number of solutions in $s,t,x,y$ by 
\begin{equation}\label{f:BP_2}
    O\left( \frac{\D N \tau}{q} + \frac{N}{q} \cdot \frac{(N/\D q)^{2d-2}}{q^{1+\d(a)}} N^{1/2 + o(1)} \right) =
    O\left( \frac{N^{2-\frac{1}{2(2d-1)} + o_d(1)}}{q^{2+\frac{\d(a)}{2d-1}}} \right) \,.
\end{equation}
Indeed, the number of possible values of $t$ is $O(N/q)$ and by definition of $s$ we know that $q^{1+\d(a)}$ divides $s$. Combining \eqref{cond:s} and choosing $\D$ to satisfy $\D^{2d-1} = N^{2d-2+1/2} /q^{2d-1+\d(a)}$ the claimed bound \eqref{f:BP_2} follows.

    We now turn to the case when polynomial $G+P_{s,t}$ has linear factors over $\C$. 
    In this case, for some $\alpha, \beta, \gamma \in \C$, one can write \eqref{f:G+P} as 
\begin{equation}\label{f:sigma_def}
    \sigma  = (\a x +\beta y+\gamma) F(x,y) 
\end{equation}
    with 
    \begin{equation}\label{tmp:03.02_1}
        F(x,y) = \a^{-1} x^{d-1} g(y) + \beta^{-1} y^{d-1} g(x) + \dots \,,
    \end{equation}
   where it is easy to check that the case  $\alpha =0$ or $\beta=0$ is not possible. 
We claim that if $2\lambda' - sd \neq 0$, then $g(x) = \lambda x^m$, $m=d-2$, $\lambda'=\lambda$ and $\gamma =0$ 
(the case $s=0$, $g=\lambda x^{m}$ corresponds to the trivial solutions and was considered before).
Indeed, since $\deg \tilde{P}_{s,t} \le d-2$ and $\deg g \le d-2$, from the definition of $\sigma$ and \eqref{tmp:03.02_1} we get
\[
    g(x) g(y) + (2\lambda' - sd) (xy)^{d-1}+ \tilde{P}_{s,t} (xy) = \a \beta^{-1} y^{d-1} x g(x) + \beta \a^{-1} x^{d-1} y g(y) 
\]
\begin{equation}\label{tmp:03.02_2}
    +
    \gamma (\a^{-1} x^{d-1} g(y) + \beta^{-1} y^{d-1} g(x)) + \dots       
\end{equation}
and hence if $g(x)$ has lower order terms we would not be able to compensate $y^{d-1} x g(x)$, $x^{d-1} y g(y)$ via the left--hand side of \eqref{tmp:03.02_2}. 
Without loss of generality, we may assume that our linear factor in \eqref{f:sigma_def} is $x=\gamma -\beta y$ (with some abuse of notation, we have changed the definition of $\beta$ and $\gamma$). 
    Consequently, we arrive at 
    \begin{equation}\label{f:add_eq}
        X^2 -tX - \beta y^2 + \gamma y +s =0 \,.
    \end{equation}
    If the last equation has a linear factor we necessarily have $\gamma^2 = \beta (t^2-4s)$ (it can be seen by computing discriminant of the conic \eqref{f:add_eq}).
    If $\gamma = 0$, then $t^2 = 4s$ and since $q<2^{-1} N^{\frac{1}{2(1+\d(a))}},$ we have that the number of solutions is $O(N/q \cdot (N/\D q^{})^{2d-2} N^{o_d(1)}),$ giving negligible contribution.\\
    If $\gamma \neq 0$, then there are two possibilities: $s=0$ and $s=2\lambda'/d.$
    Without loss of generality assume that our linear factor is $X=wy+r$. 
    Substituting the last equation into \eqref{f:add_eq}, we derive
    $w^2 = \beta$ and $r(t-r)=s$. 
    Thus we have four lines: $x=\gamma - \beta y$, $X=wy+r$, $Y= t-r- wy$ and $y=y$.

    Suppose $|\mathcal{Z}_P|>2$ and $r\not\in \{0, t\}$.
    Applying part 1) of Lemma \ref{l:lines_in_V}  with $\F= \C$ 
    and line $y=y$, we deduce that there is finite number of possibilities for $w$ and for either $r$ or $t-r$. 
    Since $r(t-r)=s,$ we obtain finite number of possibilities for $t,$ provided  $r\neq 0,t$. 
    Latter is possible only if $s=0$.
   Since $y\in [N]_{q,a},$ we get $O(N/q)$ solutions.
    
    Next we consider the case $|\mathcal{Z}_P|>2$ and $r \in \{0,t\}.$ We observe that we must have $s=0$ and there is a finite number of possibilities for $w.$
    In the case $r=0$ (similar argument works for $r=t$)  our lines are $x=\gamma - \beta y$, $X=wy$, $Y= t- wy$ and $y=y$.
    If $w=1$, then since $\gamma^2 = \beta (t^2-4s)$ and $\beta = w^2$, we have 
    $\beta =1$, $\gamma = t$ and we obtain just trivial solutions (the case $\gamma = -t$ is impossible).
    Assuming now that $w\neq 1$ (recall that $t\neq 0$ and hence $t-wy = Y \neq y$)  we 
    apply part 3) of Lemma~\ref{l:lines_in_V}  with $k=2$, $s=1.$ This concludes the proof provided $P(x)$ is not of the form \eqref{def:two_circles}. 
    Next, suppose  that $P(x)$ is of the form \eqref{def:two_circles} . 
    If, additionally
    $P(x)\in \mathcal{I},$ we recall that
    $t\neq 0$ and $\gamma \neq 0$, and therefore $w$ must be a root of unity. 
    Since our solutions are non-negative 
    rationals 
    we must have
    $w=1$ and this case  has already been considered.  
If $P(x) \notin \mathcal{I}$, then part 2) of Lemma \ref{l:lines_in_V} implies that the line $Y= t- wy$ determines a point in $\mathcal{Z}_P \times \mathcal{Z}_P.$ Thus we have finite number of possibilities for $t.$   

    It remains to consider the case $|\mathcal{Z}_P|=2$.
    In this case we have system of equations $r(t-r)=s$, $w^2 = \beta$,   $\gamma^2 = \beta (t^2-4s)$ and two intersections of our lines with $\mathcal{Z}_P \times \mathcal{Z}_P$, namely, $wz_1+r=z_2$, $t-r-wz_3=z_4$ or $wz_1+r=z_2$, $wz_3+r=z_4$ or $t-r-wz_1=z_2$, $t-r-wz_3=z_4$, where $z_i$ run over $\mathcal{Z}_P$. In either case, $w$ and all other variables are determined in a unique way (there are $O(1)$ choices for $\gamma$). The case of permutations corresponds to $s=0$, $r=0,t$ and $\gamma = t$ which produces finite number of lines.
 
   To conclude the proof, we note that if \eqref{f:add_eq} has no linear factors, then it has $N^{o_d(1)}$ solutions. 
    The number of possibilities for $s$ and $t$ is $\frac{N}{q^{2+\d(a)}} \cdot (N/\D q^{})^{2d-2}$ and thus we obtain a bound which is better than \eqref{f:BP_2}. This completes the proof.

\end{proof}

\begin{remark}
\label{r:even}
 The condition that $P(x) \neq \omega (x+c)^d$ for any pair $\omega\in \Z$, $c\in \mathbb{Q}$ is clearly necessary.  
    If, say, $P(-x)=P(x)$, then we find 
    non--trivial solutions of the form $(-z,w,z,w)$, where $z,w\in [-N,N]$ are {\it integers} (not necessary positive).
    Nevertheless, as one can see from the proof if $P$ is a generalized even polynomial, then we have similar asymptotic formula for $\EE^\times (P ([N]_{a,q}) )$ albeit with a different main term corresponding to the \enquote{generalized} trivial solutions.
\end{remark}

\begin{remark}
    
It is worth mentioning that our exponent $5/3$ in \eqref{f:E(p)_2} coincides with that of Hooley's from \cite{Hooley_5/3}, where the author considers equation $x^d +y^d = X^d+Y^d$ with  $x,y\in [N].$ More general binary forms are considered in \cite{Hooley_binary}. 
A special (multiplicative) form of our equation \eqref{tmp:1-} makes the calculations simpler. 
\end{remark}

\begin{remark}\label{r:generalizations}    
    The argument of Theorem \ref{thm: energy} is rather general and one can consider the common energy $\EE^{\times} (P([N]), S, P([N]), S)$ and even the energy $\EE^{\times} (S_1,S_2,S_3,S_4)$ for sufficiently large sets $S_i \subseteq [N]$.
   Furthermore, similar argument works for the equation 
\begin{equation}\label{eq:11=22}
    P_1 (x) P_1 (y) = P_2 (X) P_2(Y) \,,
    \quad \quad x,y,X,Y \in [N]_{a,q} \,,
\end{equation}
    where $P_1,P_2 \in \Z[x]$, $\deg(P_1)= \deg(P_2) = d>1$, $P_1(x), P_2(x) \neq \omega (x+c)^d$ for any $\omega\in \Z$, $c\in \mathbb{Q}$ and $P_1,P_2$ have the same leading coefficients. These observations will play an important role in our future work on the analogs of Theorems \ref{thm: CLT} and \ref{thm:largevaluelocal} for Rademacher random multiplicative functions.
\end{remark}

\section{Large fluctuations: proofs of Corollary~\ref{Khinchine} and Theorem~\ref{thm:largevaluelocal}  }\label{Sec: large value}
In order to prove Corollary~\ref{Khinchine} we first show that this is directly implied by the local version, that is Theorem~\ref{thm:largevaluelocal}, in the spirit of Harper's work \cite{Harperlargevalue}. 
We begin by recalling the first Borel–Cantelli lemma. 
\begin{lemma}[The first Borel–Cantelli Lemma]\label{BC}
Let $\{E_n\}_{n\ge 1}$ be any sequence of events. Then
\[\sum_{n=1}^{+\infty}\P(E_n) < +\infty \implies \P (\limsup_{n\to +\infty} E_n) = 0, \]
where 
\[\limsup_{n\to +\infty} E_n = \bigcap_{n=1}^{+\infty} \bigcup_{m=n}^{+\infty} E_m.\]
\end{lemma}

\begin{proof}[Proof of Corollary~\ref{Khinchine} assuming Theorem~\ref{thm:largevaluelocal}]
Let $W(X) = (\log \log X)^{0.02}$. Theorem~\ref{thm:largevaluelocal} implies that the probability that  \eqref{eqn: large value local} fails is at most $O(1/W(X)).$ Summing over a suitable sparse set of $X$-values, we can guarantee that the series of exceptional probabilities converges and thus by Lemma \ref{BC}, almost surely, only finitely many $X$ in the chosen sparse set make the events \eqref{eqn: large value local} fail. Consequently, there exists arbitrary large $x,$ for which \eqref{eqn: large value} holds.
\end{proof}

The rest of this section is devoted to the proof of Theorem~\ref{thm:largevaluelocal}.
We recall the following result borrowed from the work of Maynard and Rudnick \cite{MaynardRudnick}.
\begin{lemma}\label{lem: large divisor}
Let $P(x) \in \Z[x]$ be a polynomial with $\deg P\ge 2$ which is not a product of linear factors (over $\mathbb{Q}$). Then for a positive proportion of integers $n$, 
\[P^{+}(P(n)) \ge \frac{1}{2d^{2}}  \cdot  n \log n.  \]
\end{lemma}

Next we collect two important probabilistic tools ( see \cite{harper2013note, RR2009, HL2018}).

\begin{lemma}[Normal approximation result]\label{lem: normal approx}
Suppose that $m\ge 1$, and that $\mathcal{R}$ is a finite nonempty set. Suppose that for each $1 \le  i \le  m$ and $h \in  \mathcal{R}$ we are given a deterministic coefficient $c(i, h) \in \mathbb{C}$. Finally, suppose that $(V_i)_{1\le i \le m}$ is a sequence of independent, mean
zero, complex valued random variables, and let $Y = (Y_h)_{h\in \mathcal{R}}$ be the $\# \mathcal{R}$-dimensional
random vector with components $Y_h := \mathfrak{R} (
\sum_{i=1}^{m} c(i, h)V_i)$.
If $Z = (Z_h)_{h\in \mathcal{R}}$ is a multivariate normal random vector with the same mean vector
and covariance matrix as $Y$, then for any $u \in  \mathbb{R}$ and any small $\eta > 0$ we have
\[\begin{split}
\P(\max_{h\in \mathcal{R}} Y_h \le u ) & \le \P( \max_{h\in \mathcal{R}} Z_h \le u + \eta )  \\
& + O \left ( \frac{1}{\eta^{2}} \sum_{g, h\in \mathcal{R}} \sqrt{\sum_{i=1}^{m} |c(i, g)|^{2}  |c(i, h)|^{2} \E|V_i|^{4}  } + \frac{1}{\eta^{3}} \sum_{i=1}^{m} \E|V_i|^{3}( \sum_{h\in \mathcal{R}} |c(i, h)|^{3} ) 
\right)\,.
\end{split} \]
\end{lemma}

\begin{lemma}[Normal comparison result]\label{lem: normal comparison}
Suppose that $n \ge  2$, and that $\ee \ge  0 $ is sufficiently
small (i.e. less than a certain small absolute constant). Let $X_1, ..., X_n$ be mean zero,
variance one, jointly normal random variables, and suppose $\E X_iX_j \le \ee $ whenever $i \neq j$ .
Then for any $100 \ee \le  \delta  \le 1/100$ (say), we have
\[\P ( \max_{1\le i\le n }X_i \le \sqrt{(2-\delta ) \log n} )\le e^{ -\Theta(n^{\delta/20}/\sqrt{\log n}) } + n^{-\delta^{2} /50\ee} . \] 
\end{lemma}

\begin{proof}[Proof of Theorem~\ref{thm:largevaluelocal} ]
Let $X$ be large and 
$x_i = X^{i(\log 3 i)^{2}} $ for all  $1\le i  \le  \log X$ such that all of the points belong to $[X, X^{2\log X(\log  \log X)^{2}}]$ .
We aim to show that with probability $1- O(\frac{1}{W(X)})$ where $W(X) \to \infty$ as $X \to \infty$, 
\begin{equation}\label{eqn: deduction}
    \max_{1\le i \le \log X} \frac{ | \sum_{n\le x_i} f(P(n)) | }{ \sqrt{x_i  \log \log x_i }} \gg 1,
\end{equation}
where the implicit absolute constant is independent of $X$. 
To analyze \eqref{eqn: deduction}, we use a conditioning argument. Instead of simply conditioning on small primes (as has been done before), we condition on all primes which are outside of the union of the following sets $\CA_i$.
\subsection*{Step 1: construction of sets $\CA_i$.}
We first define set $\CE_i:$ 
\begin{equation}\label{eqn: condition prime}
    \CE_i: =  \{p\ge \frac{x_i \log x_i }{2d^{2}}  : \exists~n \le x_i~ s.t. ~  p|P(n)   \} \backslash \{p\ge  \frac{x_i \log x_i }{2d^{2}} : \exists~n \le x_{i-1} ~ s.t. ~  p|P(n)   \}.
\end{equation}
We claim that  
\[ x_i \ll_d |\CE_i| \ll_d x_i. \]
It follows from Lemma~\ref{lem: large divisor} that the number of $n\le x_i$ with a prime factor $p\ge \frac{1}{2d^{2}} \cdot x_i\log x_i$ is $\gg x_i$. Since $\deg P=d,$ the first set in the definition has size $ \gg_d x_i$. 
The subtracted set has cardinality at most $\ll_d x_{i-1} = o(x_i)$
yileding
$|\CE_{i}| \gg_d x_i$. The other direction is immediate as the first set in the definition is already of the size $ \ll_d x_i$.\\
Next we consider the set 
\[\CF_{i+1} : =  \CE_{i+1} \backslash \CE_i.\]
Notice that $x_{i+1}/x_i\gg X$, and so we have 
\[x_i \ll_d |\CF_{i}| \ll_d x_i, \]
and $\CF_{i}$ are disjoint. We further pick a subset $\CA_i \subset \CF_i$ defined to be the largest subset of $\CF_i$ such that no two distinct primes in $\CA_i$ both divide $P(n)$ for some $n\le x_i.$ Since $P(n)$ can have at most $d$ prime factors of size $\ge \frac{1}{2d^{2}} \cdot x_i \log x_i$ for any given $n\le x_i$ using greedy algorithm, one can pick such a set $\CA_i$ with  
\[ |\CA_i| \ge |\CF_i|/d \gg_d x_i.\]

In summary, we have chosen sets of primes $\CA_i$ for $1\le i\le \log X $ such that 
\begin{enumerate}
    \item $\CA_i \cap \CA_j = \emptyset$ for all $i \neq j$.
    \item $|\CA_i| \asymp_d x_i  $
    \item There does not exist $p, q \in \CA_i$ with $p\neq q$ such that $pq | P(n)$ for some  $n\le x_i$. 
\end{enumerate}
 We write 
 \[\CA: = \bigcup_{1\le i\le  \log X} \CA_i. \]
\subsection*{Step 2: splitting the sum and treating negligible part.}

We split the sum into the following three pieces, depending on how many prime factors in $\CA_i $ that $P(n)$ has. Let $ \sum_{n\le x_i} f(P(n))  = S_{i,1} +S_{i, 2} + S_{i,3} $ where
\[\begin{split}
     &S_{i,1} := \sum_{p\in \CA_i}  \sum_{\substack{n\le x_i\\ p|P(n)\\ q\neq p: q|P(n) \implies q \not\in \CA}} f(P(n)) \,,\\
    &S_{i, 2} :=   \sum_{ \substack{n\le x_i\\
    \exists p \in \CA\backslash \CA_i\\
    p|P(n)}} f(P(n)) \,,\\
    & S_{i,3}:= \sum_{\substack{n\le x_i \\ p|P(n) \implies p \not\in  \CA}} f(P(n)).
\end{split}
\]
By the decomposition above, we have
\begin{equation}\label{eqn: prob}
\P \left( \max_{1\le i \le \log X} \frac{ | \sum_{n\le x_i} f(P(n)) | }{ \sqrt{x_i  \log \log x_i }} \gg 1 \right) \ge \P(E_1  \cap E_2)  \ge \P(E_1) - (1- \P(E_2))
\end{equation}
where $E_1$ and $E_2$ are the events (with appropriately chosen absolute constants) 
\[E_1 : =  \max_{1\le i\le \log X} \frac{|S_{i,1} + S_{i, 3}|}{\sqrt{x_i \log \log x_i }}  \gg 1 ,  \]
\[E_2 : = \max_{1\le i\le \log X}  \frac{|S_{i,2}|}{\sqrt{x_i} (\log \log x_i )^{0.01}}  \ll 1 . \]
Our plan now goes as follows: we first use union bound to show that the event $E_2$ happens with probability close to $1$. A more subtle task is to show that $\P(E_1)$ is close to $1.$ In order to do so, we use a conditioning argument. We first show that with 
probability close to $1$ the sum $S_{i,3}$ is small (which only depends on $f(p)$ for $p \not \in \CA$ ). In particular, with probability close to $1$, we can find a large random subset of indexes $\mathcal{R}\subset \{1,2,\dots , \lfloor \log X \rfloor \}$ such that for those $i\in\mathcal{R}$ all the corresponding $S_{i,3}$ are small. We then condition on all $f(p)$ with $p\not\in \CA$ (now $\mathcal{R}$ is fixed) and thus the sum $S_{i,1}$ transforms into a sum of independent variables with certain weights. Latter puts us in the position of applying Lemma \ref{lem: normal comparison}  to produce large fluctuations.

We now estimate $\P(E_2
)$. By definition in \eqref{eqn: condition prime}, we have that if $p\in \CA \backslash \CA_i$ and $p|P(n)$ for $n\le x_i$, then one must have $p\in \cup_{1\le j\le i-1}\CA_j$.  
Recall that $x_j= X^{j(\log j)^{2}}$ and we have 
\[\begin{split}
\E|S_{i,2}|^{2} 
    & = \#\{ n \le x_i: \exists~ p \in \bigcup_{1\le j \le i-1} \CA_j ~~\text{s.t.}~~p|P(n)   \} \\
    & \ll  \sum_{ p \in \bigcup_{1\le j \le i-1} \CA_j} \left (  \frac{x_i}{p} + 1 \right )\\
    & \ll  x_i \sum_{1\le j \le i-1} \sum_{ p \in \CA_j }  \frac{1}{p} + \sum_{1\le j \le i-1} |\CA_j| \\
    & \ll x_i \sum_{1\le j \le i-1} \frac{1}{\log x_j} + \frac{x_i}{X}\\
    & \ll \frac{x_i}{\log X}.
\end{split}  \]
Using Markov's inequality, the event $|S_{i,2}| \gg \sqrt{x_i}( \log \log X)^{0.1}$ occurs with probability at most $O(1/\log X( \log \log X)^{0.2}).$ 
Applying union bound for all $1\le i \le \log X$ we get that with appropriately chosen implicit constant,
\begin{equation}\label{eqn: P(E_2)}
   \P(E_2) = \P\left(\max_{1\le i \le \log X} \frac{|S_{i,2}|} { \sqrt{x_i} (\log \log X)^{0.1}} \ll 1\right) = 1 - O(1/(\log \log X)^{0.2}).
\end{equation}
\subsection*{Step 3: creating large fluctuations.}
We next estimate $\P(E_1).$ To deal with $S_{i, 3}$ for $1\le i\le \log X$, we aim to show that with probability 
$1 - O(\frac{1}{(\log \log X)^{0.02}}),$ there exists a large random subset $\mathcal{R} \subset \{1,2, 3, \cdots, \lfloor \log X \rfloor \}$ with $|\mathcal{R}| \ge 0.99 \log X$ such that for every $i\in \mathcal{R}$, 
\begin{equation}\label{eqn
: random}
 |S_{i, 3}| = | \sum_{\substack{n\le x_i\\ p|P(n) \implies p\not \in \CA}} f(P(n)) | \ll \sqrt{x_i} (\log \log x_i)^{0.01}.
\end{equation}
Indeed, using second moment estimate and Markov's inequality, the expected number of points $x_i$ with $1\le i\le \log X$ for which \eqref{eqn
: random} fails is 
\[  \E( \#\{ i\le \log X: \eqref{eqn
: random}~\text{fails}
\})  \ll \frac{\log X}{(\log \log X)^{0.02}} ,    \]
and the claim follows. By our construction, each of the sums $S_{i,3}$ is  independent of the values $f(p)$ for $p\in\CA.$ 
From now on, we condition on all variables $f(p)$ with $p\not\in \CA$ and thus the set $\mathcal{R}$ is fixed. We aim to apply Lemma~\ref{lem: normal approx} to understand maximum of $S_{i, 1}$ over $i\in\mathcal{R}.$ Notice that for each fixed $i$, $S_{i,1}$ is a sum of independent random variables. Since by our construction all $\CA_i$ are disjoint for different choices of $i,$ we crucially have that $S_{i,1}$'s are independent. Thus, with $\tilde{\P}$ denoting the conditional probability, Lemma~\ref{lem: normal approx} implies that for any $u\in \mathbb{R}$ and small $\eta>0$, 
\begin{equation}\label{eqn: comparison}
\begin{split}
&    \tilde{\P} (\mathfrak{Re}\max_{i\in \mathcal{R}} \frac{S_{i, 1}}{\sqrt{x_i}}  \le u)  \le \P (\max_{i \in \mathcal{R}} Z_i \le u + \eta ) \\ 
    & + O\left(\frac{1}{\eta^{2} } \sum_{i\in \mathcal{R}} \sqrt{ \sum_{p\in \CA_i}  \E |\frac{1}{\sqrt{x_i}} \sum_{\substack{n\le x_i\\ p|P(n)\\ q\neq p: q|P(n) \implies q \not\in \CA}} f(P(n))|^{4} } \right) \\
    & + O\left(\frac{1}{\eta^{3}}
    \sum_{i\in \mathcal{R}}
    \sum_{p \in  \CA_i} \E  \Big| \frac{1}{\sqrt{
    x_i}} \sum_{\substack{n\le x_i\\ p|P(n)\\ q\neq p: q|P(n) \implies q \not\in \CA}} f(P(n)) \Big|^{3} \right) ,
\end{split}  
\end{equation}
where $Z_i$ are jointly normal random variables with mean zero 
\[\E (Z_i) :=  \tilde{\E} \frac{S_{i, 1}}{\sqrt{x_i}} = 0, \]
and $\tilde{\E}$ represents the conditional expectation (conditioned on all values of $f(p)$ with $p\not\in \CA$). For every $i\in\mathcal{R},$ the variance is
\[\E Z_i^{2} : = \tilde \E (\mathfrak{R}S_{i, 1})^{2} =   \frac{1}{2x_i} \sum_{p \in \CA_i} \Big| \sum_{\substack{n\le x_i\\p|P(n)\\ q\neq p:  q|P(n) \implies q \not\in \CA}} f(P(n)) \Big|^{2}.\]

The ``big $O$h" terms in \eqref{eqn: comparison} can be upper bounded by noticing that 
\begin{equation}\label{eqn: trivial bound}
    \Big| \frac{1}{\sqrt{
    x_i}} \sum_{\substack{n\le x_i\\ p|P(n)\\ q\neq p: q|P(n) \implies q \not\in \CA}} f(P(n)) \Big| \ll_d \frac{1}{\sqrt{x_i}}. 
\end{equation}
Indeed, for each fixed $p\in \CA_i$, the number of $n\le x_i$ such that $p|P(n)$ is $O_d(1)$ (since $p\gg x_i\log x_i$). Plugging \eqref{eqn: trivial bound} into \eqref{eqn: comparison}, and noticing that $x_1\ge X$, the error terms in \eqref{eqn: comparison} can be bounded by 
\[\frac{1}{\eta^{2}} \frac{|\mathcal{R}|}{X} + 
\frac{1}{\eta^{3}} \frac{|\mathcal{R}|}{\sqrt{X
}} \ll \eta^{-3} X^{-1/2 +\ee},
\] 
for any given $\ee>0$.
\subsection*{Step 4: analyzing Gaussian model.}
From now on, we only need to focus on $Z_i$ with $i\in \mathcal{R}$ and we aim to bound the probability $\P (\max_{i \in \mathcal{R}} Z_i \le u + \eta )$ for appropriately chosen $u,\eta.$ To this end,
we first show that there exist constants
$m,c>0$, such that with  probability at least $1 -O( \frac{1}{X^{c}})$  one has 
$\min_{i \in \mathcal{R}} \E Z_i^{2}\ge m.$
Let 
\[\mathcal{T}_{i, p}: = \{n\le x_i: p|P(n),  \text{ and there does not exist $q\in \CA$ such that $q\neq p$ and $q|P(n)$ }  \}.\]
Over all realizations of $f(p)$ with $p\not \in \CA $ which we conditioned on, the expected value of $\E Z_i^{2}$ is 
\[\begin{split}
\E Z_i^{2}=\mu_i  & :=  \frac{1}{2x_i} \sum_{p\in \CA_i} \#\{(n, n')\in \mathcal{T}_{i,p}^{2} : P(n)= P(n') \} \\
& \ge  \frac{1}{2x_i} \sum_{p\in \CA_i} \#\{n\in \mathcal{T}_{i,p}  \} \\ 
& \gg_d 1.    
\end{split}  \]
The last inequality follows from the definition of $\CA_i$: the number of $n\le x_i$ which have prime factors in $\CA_i$ is $\gg_d x_i$ and those $n$ for which $P(n)$ also has some prime factor $q \in \cup_{1\le k\le i-1} \CA_k$ is $o(x_i).$

Our final ingredient is the following concentration result, which directly follows from the energy estimates proved in Section~\ref{Sec: energy}. We have
\begin{equation}
\begin{split}
     & \E (\E Z_i^{2} - \mu_i)^{2} \\
     & = \E(\E Z_i^{2})^{2} -\mu_i^{2} \\
& = \frac{1}{4x_i^{2}} \sum_{p, q\in \CA_i} \#\{(n_1, n_2, n_3, n_4)\in \mathcal{T}_p^{2} \times \mathcal{T}_q^{2}: P(n_1) P(n_3) = P(n_2)P(n_4) \} -\mu_i^{2}  \\
\\ & \ll   \frac{1}{x_i^{2}} \#\{( n_1, n_2, n_3, n_4)\in[x_i]^{4}: 
  \{P(n_1),P(n_3)\}\neq \{P(n_2), P(n_4) \} ~\text{and}~P(n_1)P(n_3) = P(n_2)P(n_4) \}\\
&\ll \frac{1}{x_i^{c'}},  
\end{split}
\end{equation}
for some constant $c'$ depending on $d$. 
The last inequality follows from the power saving in
Proposition~\ref{energy}. Latter implies that the exceptional probability, i.e. 
$\E Z_i^{2
} < \mu_i/2$ is at most $O(1/x_i^{c'/2})$. Taking union bound over all $i \in \mathcal{R}$, we conclude that for some $c>0,$ with probability at least $1 -O( \frac{1}{X^{c}})$ we have that for all $ i \in \mathcal{R}$, $\E Z_i^{2}\ge m$ for an absolute constant $m$. It follows that, for any $u\in \R$ and small $\eta$, we have
\begin{equation}\label{eqn: eta}
    \P (\max_{i \in \mathcal{R}} Z_i \le u + \eta ) \le  \P \left( \max_{i \in \mathcal{R}} \frac{Z_i}{\sqrt{\E Z_i^{2}}}  \le \frac{u+\eta}{\sqrt{m}} \right) +O\left(\frac{1}{X^{c}}\right).
\end{equation}
Let $u= \sqrt{m \log \log X}$. Since $\CA_i$'s are disjoint, one has
\begin{equation}\label{eqn: covariance}
  \E Z_i Z_j =0, \quad~\text{if $i\neq j$}.
\end{equation}
Lemma~\ref{lem: normal comparison} implies that
\begin{equation}\label{eqn: eta 2}
    \P \left( \max_{i \in \mathcal{R}} \frac{Z_i}{\sqrt{\E Z_i^{2}}}  \le \frac{ \sqrt{m \log \log X}+\eta}{\sqrt{m}} \right) \ll (\log X)^{-\Theta((\log X)^{1/3000})}.
\end{equation}
Plugging \eqref{eqn: eta} and \eqref{eqn: eta 2} into \eqref{eqn: comparison} and choosing $\eta$ to be a fixed constant,
 we derive
\[\tilde\P \left( \max_{ i \in \mathcal{R}  } \mathfrak{Re}\frac{S_{i,1}}{\sqrt{x_i }} \le 
\sqrt{m \log \log X}  \right) \ll (\log X)^{-\Theta((\log X)^{1/3000})}.   \]
Since $\log \log  x_i \ll \log \log X^{i (\log 3i)^{2}} \ll \log \log X$ for all $i\le \log X$, latter inequality can be rewritten as
\begin{equation}\label{eqn: P(E_1)}
    \tilde\P \left( \max_{ i \in \mathcal{R}  } \mathfrak{Re}\frac{S_{i,1}}{\sqrt{x_i  \log \log x_i }}  \gg 1  \right) \ge 1 - O\left( (\log X)^{-\Theta((\log X)^{1/3000})} \right), 
\end{equation}
for an appropriately chosen small absolute constant. Since the probability of existence of $\mathcal{R}$ satisfying \eqref{eqn
: random} is at least $1- O(\frac{1}{(\log \log x)^{0.02}}),$ we can combine this with the error term in \eqref{eqn: P(E_1)} to arrive at the estimate
\begin{equation}\label{eqn: P(E_1) final}
    \P(E_1) \ge 1- O\left(\frac{1}{(\log \log x)^{0.02}}\right).
\end{equation}
Inserting \eqref{eqn: P(E_2)}, \eqref{eqn: P(E_1) final}
into \eqref{eqn: prob} concludes the proof.
\end{proof}

\bibliographystyle{plain}
\bibliography{chowla}{}
\nocite{BDW}

\end{document}